\documentclass{amsart}
\usepackage{amsfonts}

\newtheorem{theorem}{Theorem}
\theoremstyle{plain}

\newtheorem{corollary}{Corollary}

\newtheorem{definition}{Definition}
\newtheorem{example}{Example}

\newtheorem{lemma}{Lemma}

\newtheorem{proposition}{Proposition}

\numberwithin{equation}{section}
\input{tcilatex}

\begin{document}
\title[integral inequalities for $s$-geometrically convex functions]{On some
integral inequalities for $s$-geometrically convex functions and their
applications}
\author{Mevl\"{u}t TUN\c{C}}
\address{Department of Mathematics, Faculty of Science and Arts, University
of Kilis 7 Aral\i k, 79000, Kilis, Turkey}
\email{mevluttunc@kilis.edu.tr}
\date{October 08, 2012}
\subjclass[2000]{Primary 26D10, 26D15}
\keywords{geometrically convex, s-geometrically convex, h\"{o}lder
inequality, power mean inequality}
\thanks{This paper is in final form and no version of it will be submitted
for publication elsewhere.}

\begin{abstract}
In this paper, we establish three inequalities for differentiable $s$%
-geometrically and geometrically convex functions which are connected with
the famous Hermite-Hadamard inequality holding for convex functions. Some
applications to special means of positive real numbers are given.
\end{abstract}

\maketitle

\section{Introduction}

In this section we will present definitions and some results used in this
paper.

\begin{definition}
Let $I$ be an interval in $%
%TCIMACRO{\U{211d} }%
%BeginExpansion
\mathbb{R}
%EndExpansion
.$ Then $f:I\rightarrow 
%TCIMACRO{\U{211d} }%
%BeginExpansion
\mathbb{R}
%EndExpansion
,\emptyset \neq I\subseteq 
%TCIMACRO{\U{211d} }%
%BeginExpansion
\mathbb{R}
%EndExpansion
$\ is said to be convex if 
\begin{equation}
f\left( tx+\left( 1-t\right) y\right) \leq tf\left( x\right) +\left(
1-t\right) f\left( y\right) .  \label{1}
\end{equation}%
for all $x,y\in I$ and $t\in \left[ 0,1\right] .$
\end{definition}

\begin{definition}
\cite{hud}\textit{\ Let }$s\in \left( 0,1\right] .$\textit{\ A function }$%
f:I\subset 
%TCIMACRO{\U{211d} }%
%BeginExpansion
\mathbb{R}
%EndExpansion
_{0}=\left[ 0,\infty \right) \rightarrow 
%TCIMACRO{\U{211d} }%
%BeginExpansion
\mathbb{R}
%EndExpansion
_{0}$\textit{\ is said to be }$s-$\textit{convex in the second sense if \ \
\ \ \ \ \ \ \ \ \ \ }%
\begin{equation}
f\left( tx+\left( 1-t\right) y\right) \leq t^{s}f\left( x\right) +\left(
1-t\right) ^{s}f\left( y\right)  \label{2}
\end{equation}%
\textit{for all }$x,y\in I$\textit{\ and }$t\in \left[ 0,1\right] $\textit{.}
\end{definition}

It can be easily checked for $s=1$, $s$-convexity reduces to the ordinary
convexity of functions defined on $\left[ 0,\infty \right) $.

\bigskip Recently, In \cite{zh}, the concept of geometrically and $s$%
-geometrically convex functions was introduced as follows.

\begin{definition}
\cite{zh} A function $f:I\subset 
%TCIMACRO{\U{211d} }%
%BeginExpansion
\mathbb{R}
%EndExpansion
_{+}=\left( 0,\infty \right) \rightarrow 
%TCIMACRO{\U{211d} }%
%BeginExpansion
\mathbb{R}
%EndExpansion
_{+}$\textit{\ }is said to be a geometrically convex function if%
\begin{equation}
f\left( x^{t}y^{1-t}\right) \leq \left[ f\left( x\right) \right] ^{t}\left[
f\left( y\right) \right] ^{1-t}  \label{3}
\end{equation}%
for \textit{all }$x,y\in I$\textit{\ and }$t\in \left[ 0,1\right] $\textit{.}
\end{definition}

\begin{definition}
\cite{zh} A function $f:I\subset 
%TCIMACRO{\U{211d} }%
%BeginExpansion
\mathbb{R}
%EndExpansion
_{+}\rightarrow 
%TCIMACRO{\U{211d} }%
%BeginExpansion
\mathbb{R}
%EndExpansion
_{+}$\textit{\ }is said to be a $s$-geometrically convex function if%
\begin{equation}
f\left( x^{t}y^{1-t}\right) \leq \left[ f\left( x\right) \right] ^{t^{s}}%
\left[ f\left( y\right) \right] ^{\left( 1-t\right) ^{s}}  \label{4}
\end{equation}%
for some $s\in \left( 0,1\right] $, where $x,y\in I$\textit{\ and }$t\in %
\left[ 0,1\right] $\textit{.}
\end{definition}

If $s=1$, the $s$-geometrically convex function becomes a geometrically
convex function on $%
%TCIMACRO{\U{211d} }%
%BeginExpansion
\mathbb{R}
%EndExpansion
_{+}$.

\begin{example}
\bigskip \cite{zh} Let $f\left( x\right) =x^{s}/s,$ $x\in \left( 0,1\right]
, $ $0<s<1,$ $q\geq 1,$ and then the function%
\begin{equation}
\left\vert f^{\prime }\left( x\right) \right\vert ^{q}=x^{\left( s-1\right)
q}  \label{5}
\end{equation}%
is monotonically decreasing on $\left( 0,1\right] $. For $t\in \left[ 0,1%
\right] $, we have%
\begin{equation}
\left( s-1\right) q\left( t^{s}-t\right) \leq 0,\text{ \ \ }\left(
s-1\right) q\left( \left( 1-t\right) ^{s}-\left( 1-t\right) \right) \leq 0.
\label{6}
\end{equation}%
Hence, $\left\vert f^{\prime }\left( x\right) \right\vert ^{q}$ is $s$%
-geometrically convex on $\left( 0,1\right] $ for $0<s<1$.
\end{example}

\bigskip In \cite{dr}, the following Lemma and its related Hermite-Hadamard
type inequalities for convex functions were obtained.

\begin{lemma}
\label{l1}\cite{dr} Let $f:I^{\circ }\subseteq 
%TCIMACRO{\U{211d} }%
%BeginExpansion
\mathbb{R}
%EndExpansion
\rightarrow 
%TCIMACRO{\U{211d} }%
%BeginExpansion
\mathbb{R}
%EndExpansion
$ be a differentiable mapping on $I^{\circ },$ $a,b\in I^{\circ }$ with $a<b$%
. If $f^{\prime }\in L\left[ a,b\right] ,$ then the following equality holds:%
\begin{equation}
\frac{f\left( a\right) +f\left( b\right) }{2}-\frac{1}{b-a}%
\int_{a}^{b}f\left( x\right) dx=\frac{b-a}{2}\int_{0}^{1}\left( 1-2t\right)
f^{\prime }\left( ta+\left( 1-t\right) b\right) dt.  \label{7}
\end{equation}
\end{lemma}

\begin{theorem}
\cite{dr} Let $f:I^{\circ }\subseteq 
%TCIMACRO{\U{211d} }%
%BeginExpansion
\mathbb{R}
%EndExpansion
\rightarrow 
%TCIMACRO{\U{211d} }%
%BeginExpansion
\mathbb{R}
%EndExpansion
$ be a differentiable mapping on $I^{\circ },$ $a,b\in I^{\circ }$ with $a<b$%
. If $\left\vert f^{\prime }\right\vert $ is convex on $\left[ a,b\right] ,$
then the following inequality holds:%
\begin{equation}
\left\vert \frac{f\left( a\right) +f\left( b\right) }{2}-\frac{1}{b-a}%
\int_{a}^{b}f\left( x\right) dx\right\vert \leq \frac{\left( b-a\right)
\left( \left\vert f^{\prime }\left( a\right) \right\vert +\left\vert
f^{\prime }\left( b\right) \right\vert \right) }{8}.  \label{8}
\end{equation}
\end{theorem}

\begin{theorem}
\cite{dr} Let $f:I^{\circ }\subseteq 
%TCIMACRO{\U{211d} }%
%BeginExpansion
\mathbb{R}
%EndExpansion
\rightarrow 
%TCIMACRO{\U{211d} }%
%BeginExpansion
\mathbb{R}
%EndExpansion
$ be a differentiable mapping on $I^{\circ },$ $a,b\in I^{\circ }$ with $a<b$%
, and let $p>1.$ If the mapping $\left\vert f^{\prime }\right\vert
^{p/\left( p-1\right) }$ is convex on $\left[ a,b\right] ,$ then the
following inequality holds:%
\begin{equation}
\left\vert \frac{f\left( a\right) +f\left( b\right) }{2}-\frac{1}{b-a}%
\int_{a}^{b}f\left( x\right) dx\right\vert \leq \frac{b-a}{2\left(
p+1\right) ^{1/p}}\left[ \frac{\left\vert f^{\prime }\left( a\right)
\right\vert ^{p/\left( p-1\right) }+\left\vert f^{\prime }\left( b\right)
\right\vert ^{p/\left( p-1\right) }}{2}\right] ^{\left( p-1\right) /p}.
\label{9}
\end{equation}
\end{theorem}

\bigskip The goal of this paper is to establish some inequalities of
Hermite-Hadamard type for geometrically and $s$-geometrically convex
functions.

\section{On some inequalities for $s$-geometrically convexity}

\begin{theorem}
\label{t1}Let $f:I\subseteq 
%TCIMACRO{\U{211d} }%
%BeginExpansion
\mathbb{R}
%EndExpansion
_{+}\rightarrow 
%TCIMACRO{\U{211d} }%
%BeginExpansion
\mathbb{R}
%EndExpansion
_{+}$ be a differentiable mapping on $I^{\circ },$ $a,b\in I^{\circ }$ with $%
a<b$, and $f^{\prime }\in L\left[ a,b\right] .$ If $\left\vert f^{\prime
}\right\vert $ is $s$-geometrically convex and monotonically decreasing on $%
\left[ a,b\right] $ for $s\in \left( 0,1\right] ,$ then the following
inequality holds:%
\begin{equation}
\left\vert \frac{f\left( a\right) +f\left( b\right) }{2}-\frac{1}{b-a}%
\int_{a}^{b}f\left( x\right) dx\right\vert \leq \frac{b-a}{2}G_{1}\left(
s;g_{1}\left( \alpha \right) ,g_{2}\left( \alpha \right) \right)  \label{10}
\end{equation}%
where%
\begin{equation}
g_{1}\left( \alpha \right) =\left\{ 
\begin{array}{cc}
\frac{1}{4} & \alpha =1 \\ 
\frac{2\alpha ^{1/2}-2-\ln \alpha }{\left( \ln \alpha \right) ^{2}} & \alpha
\neq 1%
\end{array}%
\right. ,\text{ \ }g_{2}\left( \alpha \right) =\left\{ 
\begin{array}{cc}
\frac{1}{4} & \alpha =1 \\ 
\frac{2\alpha ^{1/2}-2\alpha +\alpha \ln \alpha }{\left( \ln \alpha \right)
^{2}} & \alpha \neq 1%
\end{array}%
\right.  \label{t}
\end{equation}%
\begin{equation*}
\alpha \left( u,v\right) =\left\vert f^{\prime }\left( a\right) \right\vert
^{u}\left\vert f^{\prime }\left( b\right) \right\vert ^{-v},\text{ }u,v>0,
\end{equation*}%
\begin{equation*}
G_{1}\left( s;g_{1}\left( \alpha \right) ,g_{2}\left( \alpha \right) \right)
=\left\vert f^{\prime }\left( b\right) \right\vert ^{s}\left[ g_{1}\left(
\alpha \left( s,s\right) \right) +g_{2}\left( \alpha \left( s,s\right)
\right) \right] ,\text{ }\left\vert f^{\prime }\left( a\right) \right\vert
\leq 1.
\end{equation*}
\end{theorem}

\begin{proof}
Since $\left\vert f^{\prime }\right\vert $ is $s$-geometrically convex and
monotonically decreasing on $\left[ a,b\right] $, from Lemma \ref{l1}, we
have%
\begin{eqnarray*}
&&\left\vert \frac{f\left( a\right) +f\left( b\right) }{2}-\frac{1}{b-a}%
\int_{a}^{b}f\left( x\right) dx\right\vert \\
&\leq &\left\vert \frac{b-a}{2}\int_{0}^{1}\left( 1-2t\right) f^{\prime
}\left( ta+\left( 1-t\right) b\right) dt\right\vert \\
&\leq &\frac{b-a}{2}\int_{0}^{1}\left\vert 1-2t\right\vert \left\vert
f^{\prime }\left( ta+\left( 1-t\right) b\right) \right\vert dt \\
&\leq &\frac{b-a}{2}\left\{ \int_{0}^{\frac{1}{2}}\left( 1-2t\right)
\left\vert f^{\prime }\left( a^{t}b^{1-t}\right) \right\vert dt+\int_{\frac{1%
}{2}}^{1}\left( 2t-1\right) \left\vert f^{\prime }\left( a^{t}b^{1-t}\right)
\right\vert dt\right\} \\
&\leq &\frac{b-a}{2}\left\{ \int_{0}^{\frac{1}{2}}\left( 1-2t\right)
\left\vert f^{\prime }\left( a\right) \right\vert ^{t^{s}}\left\vert
f^{\prime }\left( b\right) \right\vert ^{\left( 1-t\right) ^{s}}dt+\int_{%
\frac{1}{2}}^{1}\left( 2t-1\right) \left\vert f^{\prime }\left( a\right)
\right\vert ^{t^{s}}\left\vert f^{\prime }\left( b\right) \right\vert
^{\left( 1-t\right) ^{s}}dt\right\} .
\end{eqnarray*}%
If $0<\mu \leq 1,$ $0<\alpha ,s\leq 1,$ then%
\begin{equation}
\mu ^{\alpha ^{s}}\leq \mu ^{\alpha s}.  \label{a}
\end{equation}%
If $\left\vert f^{\prime }\left( a\right) \right\vert \leq 1$, by $\left( %
\ref{a}\right) $, we get that%
\begin{eqnarray}
&&\int_{0}^{\frac{1}{2}}\left( 1-2t\right) \left\vert f^{\prime }\left(
a\right) \right\vert ^{t^{s}}\left\vert f^{\prime }\left( b\right)
\right\vert ^{\left( 1-t\right) ^{s}}dt+\int_{\frac{1}{2}}^{1}\left(
2t-1\right) \left\vert f^{\prime }\left( a\right) \right\vert
^{t^{s}}\left\vert f^{\prime }\left( b\right) \right\vert ^{\left(
1-t\right) ^{s}}dt  \notag \\
&\leq &\int_{0}^{\frac{1}{2}}\left( 1-2t\right) \left\vert f^{\prime }\left(
a\right) \right\vert ^{st}\left\vert f^{\prime }\left( b\right) \right\vert
^{s\left( 1-t\right) }dt+\int_{\frac{1}{2}}^{1}\left( 2t-1\right) \left\vert
f^{\prime }\left( a\right) \right\vert ^{st}\left\vert f^{\prime }\left(
b\right) \right\vert ^{s\left( 1-t\right) }dt  \notag \\
&=&\int_{0}^{\frac{1}{2}}\left( 1-2t\right) \left\vert f^{\prime }\left(
b\right) \right\vert ^{s}\left\vert \frac{f^{\prime }\left( a\right) }{%
f^{\prime }\left( b\right) }\right\vert ^{st}dt+\int_{\frac{1}{2}}^{1}\left(
2t-1\right) \left\vert f^{\prime }\left( b\right) \right\vert ^{s}\left\vert 
\frac{f^{\prime }\left( a\right) }{f^{\prime }\left( b\right) }\right\vert
^{st}dt  \notag \\
&=&\left\vert f^{\prime }\left( b\right) \right\vert ^{s}\left[ g_{1}\left(
\alpha \left( s,s\right) \right) +g_{2}\left( \alpha \left( s,s\right)
\right) \right]  \label{x}
\end{eqnarray}%
Thus, immediately gives the required inequality (\ref{10}).
\end{proof}

\begin{theorem}
\bigskip \label{t2}Let $f:I\subseteq 
%TCIMACRO{\U{211d} }%
%BeginExpansion
\mathbb{R}
%EndExpansion
_{+}\rightarrow 
%TCIMACRO{\U{211d} }%
%BeginExpansion
\mathbb{R}
%EndExpansion
_{+}$ be a differentiable mapping on $I^{\circ },$ $a,b\in I^{\circ }$ with $%
a<b$, and $f^{\prime }\in L\left[ a,b\right] .$ If $\left\vert f^{\prime
}\right\vert ^{q}$ is $s$-geometrically convex and monotonically decreasing
on $\left[ a,b\right] $ for $1/p+1/q=1$ and $s\in \left( 0,1\right] ,$ then
the following inequality holds:%
\begin{equation}
\left\vert \frac{f\left( a\right) +f\left( b\right) }{2}-\frac{1}{b-a}%
\int_{a}^{b}f\left( x\right) dx\right\vert \leq \frac{b-a}{2\left(
p+1\right) ^{1/p}}G_{2}\left( s,q;g_{3}\left( \alpha \right) \right)
\label{11}
\end{equation}%
where%
\begin{equation}
g_{3}\left( \alpha \right) =\left\{ 
\begin{array}{cc}
1, & \alpha =1, \\ 
\frac{\alpha -1}{\ln \alpha }, & \alpha \neq 1,%
\end{array}%
\right.  \label{12}
\end{equation}%
\begin{equation}
G_{2}\left( s,q;g_{3}\left( \alpha \right) \right) =\left\vert f^{\prime
}\left( b\right) \right\vert ^{s}\left[ g_{3}\left( \alpha \left(
sq,sq\right) \right) \right] ^{\frac{1}{q}},\text{ \ }\left\vert f^{\prime
}\left( a\right) \right\vert \leq 1.  \label{13}
\end{equation}
\end{theorem}

\begin{proof}
\bigskip Since $\left\vert f^{\prime }\right\vert ^{q}$ is $s$-geometrically
convex and monotonically decreasing on $\left[ a,b\right] $, from Lemma \ref%
{l1} and H\"{o}lder inequality, we have%
\begin{eqnarray}
&&\left\vert \frac{f\left( a\right) +f\left( b\right) }{2}-\frac{1}{b-a}%
\int_{a}^{b}f\left( x\right) dx\right\vert  \label{vv} \\
&\leq &\frac{b-a}{2}\int_{0}^{1}\left\vert 1-2t\right\vert \left\vert
f^{\prime }\left( ta+\left( 1-t\right) b\right) \right\vert dt  \notag \\
&\leq &\frac{b-a}{2}\left( \int_{0}^{1}\left\vert 1-2t\right\vert
^{p}dt\right) ^{\frac{1}{p}}\left( \int_{0}^{1}\left\vert f^{\prime }\left(
ta+\left( 1-t\right) b\right) \right\vert ^{q}dt\right) ^{\frac{1}{q}} 
\notag
\end{eqnarray}%
Using the properties of $\left\vert f^{\prime }\right\vert ^{q},$ we obtain
that%
\begin{eqnarray}
\left( \int_{0}^{1}\left\vert f^{\prime }\left( ta+\left( 1-t\right)
b\right) \right\vert ^{q}dt\right) ^{\frac{1}{q}} &\leq &\left(
\int_{0}^{1}\left\vert f^{\prime }\left( a^{t}b^{1-t}\right) \right\vert
^{q}dt\right) ^{\frac{1}{q}}  \notag \\
&\leq &\left( \int_{0}^{1}\left\vert f^{\prime }\left( a\right) \right\vert
^{qt^{s}}\left\vert f^{\prime }\left( b\right) \right\vert ^{q\left(
1-t\right) ^{s}}dt\right) ^{\frac{1}{q}}.  \label{l}
\end{eqnarray}%
If $\left\vert f^{\prime }\left( a\right) \right\vert \leq 1$, by (\ref{a}),
we get that%
\begin{eqnarray}
\left( \int_{0}^{1}\left\vert f^{\prime }\left( a\right) \right\vert
^{qt^{s}}\left\vert f^{\prime }\left( b\right) \right\vert ^{q\left(
1-t\right) ^{s}}dt\right) ^{\frac{1}{q}} &\leq &\left(
\int_{0}^{1}\left\vert f^{\prime }\left( a\right) \right\vert
^{sqt}\left\vert f^{\prime }\left( b\right) \right\vert ^{sq\left(
1-t\right) }dt\right) ^{\frac{1}{q}}  \notag \\
&=&\left( \left\vert f^{\prime }\left( b\right) \right\vert
^{sq}\int_{0}^{1}\left\vert \frac{f^{\prime }\left( a\right) }{f^{\prime
}\left( b\right) }\right\vert ^{sqt}dt\right) ^{\frac{1}{q}}  \notag \\
&=&\left\vert f^{\prime }\left( b\right) \right\vert ^{s}\left[ g_{3}\left(
\alpha \left( sq,sq\right) \right) \right] ^{\frac{1}{q}}.  \label{m}
\end{eqnarray}%
Further, since%
\begin{equation}
\int_{0}^{1}\left\vert 1-2t\right\vert ^{p}dt=\int_{0}^{\frac{1}{2}}\left(
1-2t\right) ^{p}dt+\int_{\frac{1}{2}}^{1}\left( 2t-1\right)
^{p}dt=2\int_{0}^{\frac{1}{2}}\left( 1-2t\right) ^{p}dt=\frac{1}{p+1}
\label{p}
\end{equation}%
a combination of (\ref{vv})-(\ref{p}) immediately gives the proof of
inequality (\ref{11}).
\end{proof}

\begin{corollary}
\bigskip Let $f:I\subseteq \left( 0,\infty \right) \rightarrow \left(
0,\infty \right) $ be differentiable on $I^{\circ },$ $a,b\in I$ with $a<b,$
and $f^{\prime }\in L\left( \left[ a,b\right] \right) .$ If $\left\vert
f^{\prime }\right\vert ^{q}$ is $s$-geometrically convex and monotonically
decreasing on $\left[ a,b\right] $ for $s\in \left( 0,1\right] ,$ then

i) When $p=q=2$, one has%
\begin{equation*}
\left\vert \frac{f\left( a\right) +f\left( b\right) }{2}-\frac{1}{b-a}%
\int_{a}^{b}f\left( x\right) dx\right\vert \leq \frac{b-a}{2\sqrt{2}}%
G_{2}\left( s,2,g_{3}\left( \alpha \right) \right)
\end{equation*}

ii) If we take $s=1$ in (\ref{11}), we have for geometrically convex, one has%
\begin{equation*}
\left\vert \frac{f\left( a\right) +f\left( b\right) }{2}-\frac{1}{b-a}%
\int_{a}^{b}f\left( x\right) dx\right\vert \leq \frac{b-a}{2\left(
p+1\right) ^{\frac{1}{p}}}G_{2}\left( 1,q,g_{3}\left( \alpha \right) \right)
\end{equation*}%
where $g_{3},G_{2}$ are same with (\ref{12}), (\ref{13}).\bigskip
\end{corollary}

\begin{theorem}
\label{t3}\bigskip \bigskip Let $f:I\subseteq 
%TCIMACRO{\U{211d} }%
%BeginExpansion
\mathbb{R}
%EndExpansion
_{+}\rightarrow 
%TCIMACRO{\U{211d} }%
%BeginExpansion
\mathbb{R}
%EndExpansion
_{+}$ be a differentiable mapping on $I^{\circ },$ $a,b\in I^{\circ }$ with $%
a<b$, and $f^{\prime }\in L\left[ a,b\right] .$ If $\left\vert f^{\prime
}\right\vert ^{q}$ is $s$-geometrically convex and monotonically decreasing
on $\left[ a,b\right] $ for $q\geq 1$ and $s\in \left( 0,1\right] ,$ then
the following inequality holds:%
\begin{equation}
\left\vert \frac{f\left( a\right) +f\left( b\right) }{2}-\frac{1}{b-a}%
\int_{a}^{b}f\left( x\right) dx\right\vert \leq \frac{b-a}{2}\left( \frac{1}{%
4}\right) ^{1-\frac{1}{q}}G_{3}\left( s,q;g_{1}\left( \alpha \right)
,g_{2}\left( \alpha \right) \right)  \label{111}
\end{equation}%
where $g_{1}\left( \alpha \right) ,g_{2}\left( \alpha \right) $ is the same
as in (\ref{t}), and%
\begin{eqnarray*}
&&G_{3}\left( s,q;g_{1}\left( \alpha \right) ,g_{2}\left( \alpha \right)
\right) \\
&=&\left\vert f^{\prime }\left( b\right) \right\vert ^{s}\left[ \left[
g_{1}\left( \alpha \left( sq,sq\right) \right) \right] ^{\frac{1}{q}}+\left[
g_{2}\left( \alpha \left( sq,sq\right) \right) \right] ^{\frac{1}{q}}\right]
,\text{ \ \ }\left\vert f^{\prime }\left( a\right) \right\vert \leq 1
\end{eqnarray*}
\end{theorem}

\begin{proof}
Since $\left\vert f^{\prime }\right\vert $ $^{q}$ is $s$-geometrically
convex and monotonically decresing on $\left[ a,b\right] ,$ from Lemma \ref%
{l1} and well known power mean inequality, we have%
\begin{eqnarray*}
&&\left\vert \frac{f\left( a\right) +f\left( b\right) }{2}-\frac{1}{b-a}%
\int_{a}^{b}f\left( x\right) dx\right\vert \\
&\leq &\frac{b-a}{2}\int_{0}^{1}\left\vert 1-2t\right\vert \left\vert
f^{\prime }\left( ta+\left( 1-t\right) b\right) \right\vert dt \\
&\leq &\frac{b-a}{2}\left[ \int_{0}^{\frac{1}{2}}\left( 1-2t\right)
\left\vert f^{\prime }\left( ta+\left( 1-t\right) b\right) \right\vert
dt+\int_{\frac{1}{2}}^{1}\left( 2t-1\right) \left\vert f^{\prime }\left(
ta+\left( 1-t\right) b\right) \right\vert dt\right] \\
&\leq &\frac{b-a}{2}\left[ \left( \int_{0}^{\frac{1}{2}}\left( 1-2t\right)
dt\right) ^{1-\frac{1}{q}}\left[ \int_{0}^{\frac{1}{2}}\left( 1-2t\right)
\left\vert f^{\prime }\left( ta+\left( 1-t\right) b\right) \right\vert ^{q}dt%
\right] ^{\frac{1}{q}}\right. \\
&&+\left. \left( \int_{\frac{1}{2}}^{1}\left( 2t-1\right) dt\right) ^{1-%
\frac{1}{q}}\left[ \int_{\frac{1}{2}}^{1}\left( 2t-1\right) \left\vert
f^{\prime }\left( ta+\left( 1-t\right) b\right) \right\vert ^{q}dt\right] ^{%
\frac{1}{q}}\right]
\end{eqnarray*}%
\begin{eqnarray}
&\leq &\frac{b-a}{2}\left( \frac{1}{4}\right) ^{1-\frac{1}{q}}\left[ \left[
\int_{0}^{\frac{1}{2}}\left( 1-2t\right) \left\vert f^{\prime }\left(
a^{t}b^{1-t}\right) \right\vert ^{q}dt\right] ^{\frac{1}{q}}\right.  \notag
\\
&&+\left. \left[ \int_{\frac{1}{2}}^{1}\left( 2t-1\right) \left\vert
f^{\prime }\left( a^{t}b^{1-t}\right) \right\vert ^{q}dt\right] ^{\frac{1}{q}%
}\right]  \notag \\
&\leq &\frac{b-a}{2}\left( \frac{1}{4}\right) ^{1-\frac{1}{q}}\left[ \left[
\int_{0}^{\frac{1}{2}}\left( 1-2t\right) \left\vert f^{\prime }\left(
a\right) \right\vert ^{qt^{s}}\left\vert f^{\prime }\left( b\right)
\right\vert ^{q\left( 1-t\right) ^{s}}dt\right] ^{\frac{1}{q}}\right.  \notag
\\
&&+\left. \left[ \int_{\frac{1}{2}}^{1}\left( 2t-1\right) \left\vert
f^{\prime }\left( a\right) \right\vert ^{qt^{s}}\left\vert f^{\prime }\left(
b\right) \right\vert ^{q\left( 1-t\right) ^{s}}dt\right] ^{\frac{1}{q}}%
\right]  \label{112}
\end{eqnarray}%
If $\left\vert f^{\prime }\left( a\right) \right\vert \leq 1$, by (\ref{a}),
we get that%
\begin{eqnarray}
&&\int_{0}^{\frac{1}{2}}\left( 1-2t\right) \left\vert f^{\prime }\left(
a\right) \right\vert ^{qt^{s}}\left\vert f^{\prime }\left( b\right)
\right\vert ^{q\left( 1-t\right) ^{s}}dt  \notag \\
&\leq &\int_{0}^{\frac{1}{2}}\left( 1-2t\right) \left\vert f^{\prime }\left(
a\right) \right\vert ^{sqt}\left\vert f^{\prime }\left( b\right) \right\vert
^{sq\left( 1-t\right) }dt=\left\vert f^{\prime }\left( b\right) \right\vert
^{sq}g_{1}\left( \alpha \left( sq,sq\right) \right)  \notag \\
&&\int_{\frac{1}{2}}^{1}\left( 2t-1\right) \left\vert f^{\prime }\left(
a\right) \right\vert ^{qt^{s}}\left\vert f^{\prime }\left( b\right)
\right\vert ^{q\left( 1-t\right) ^{s}}dt  \label{113} \\
&\leq &\int_{\frac{1}{2}}^{1}\left( 2t-1\right) \left\vert f^{\prime }\left(
a\right) \right\vert ^{sqt}\left\vert f^{\prime }\left( b\right) \right\vert
^{sq\left( 1-t\right) }dt=\left\vert f^{\prime }\left( b\right) \right\vert
^{sq}g_{2}\left( \alpha \left( sq,sq\right) \right)  \notag
\end{eqnarray}%
By combining of (\ref{112})-(\ref{113}) immediately gives the required
inequality (\ref{111}).
\end{proof}

\begin{corollary}
\bigskip \bigskip Let $f:I\subseteq \left( 0,\infty \right) \rightarrow
\left( 0,\infty \right) $ be differentiable on $I^{\circ },$ $a,b\in I$ with 
$a<b,$ and $f^{\prime }\in L\left( \left[ a,b\right] \right) .$ If $%
\left\vert f^{\prime }\right\vert ^{q}$ is $s$-geometrically convex and
monotonically decreasing on $\left[ a,b\right] $ for $q\geq 1,$ and $s\in
\left( 0,1\right] ,$ then

i) If we take $q=1$ in (\ref{111}), we obtain that%
\begin{equation*}
\left\vert \frac{f\left( a\right) +f\left( b\right) }{2}-\frac{1}{b-a}%
\int_{a}^{b}f\left( x\right) dx\right\vert \leq \frac{b-a}{2}G_{3}\left(
s,1;g_{1}\left( \alpha \right) ,g_{2}\left( \alpha \right) \right)
\end{equation*}

ii) If we take $s=1$ in (\ref{111}), for geometrically convex, we obtain that%
\begin{equation*}
\left\vert \frac{f\left( a\right) +f\left( b\right) }{2}-\frac{1}{b-a}%
\int_{a}^{b}f\left( x\right) dx\right\vert \leq \frac{b-a}{2}\left( \frac{1}{%
4}\right) ^{1-\frac{1}{q}}G_{3}\left( 1,q;g_{1}\left( \alpha \right)
,g_{2}\left( \alpha \right) \right)
\end{equation*}%
where $g_{1}\left( \alpha \right) ,g_{2}\left( \alpha \right) ,\alpha \left(
u,v\right) ,G_{3}\left( s,q;g_{1}\left( \alpha \right) ,g_{2}\left( \alpha
\right) \right) $ are same with above.
\end{corollary}

\section{Applications to some special means}

Let%
\begin{eqnarray*}
A\left( a,b\right) &=&\frac{a+b}{2},\text{ }L\left( a,b\right) =\frac{b-a}{%
\ln b-\ln a}\text{ \ \ \ \ }\left( a\neq b\right) , \\
L_{p}\left( a,b\right) &=&\left( \frac{b^{p+1}-a^{p+1}}{\left( p+1\right)
\left( b-a\right) }\right) ^{1/p},\text{ }a\neq b,\text{ }p\in 
%TCIMACRO{\U{211d} }%
%BeginExpansion
\mathbb{R}
%EndExpansion
,\text{ }p\neq -1,0
\end{eqnarray*}%
be the arithmetic, logarithmic, generalized logarithmic means for $a,b>0$
respectively.

\begin{proposition}
Let $0<a<b\leq 1,$ $0<s<1.$ Then%
\begin{eqnarray}
&&\left\vert A\left( a^{s},b^{s}\right) -\left[ L_{s}\left( a,b\right) %
\right] ^{s}\right\vert  \notag \\
&\leq &\frac{\left( b-a\right) sb^{^{s\left( s-1\right) }}}{2}L\left(
a^{^{s\left( s-1\right) }},b^{^{s\left( s-1\right) }}\right) \left[ A\left(
a^{^{s\left( s-1\right) }},b^{^{s\left( s-1\right) }}\right) -\left(
1/2\right) L\left( a^{^{s\left( s-1\right) }},b^{^{s\left( s-1\right)
}}\right) \right]  \label{41}
\end{eqnarray}
\end{proposition}

\begin{proof}
\bigskip The proof is obvious from Theorem \ref{t1} applied $f\left(
x\right) =x^{s}/s,$ $x\in \left( 0,1\right] ,$ $0<s<1.$ Then $\left\vert
f^{\prime }\left( a\right) \right\vert =a^{s-1}>b^{s-1}=\left\vert f^{\prime
}\left( b\right) \right\vert \geq 1$ and%
\begin{equation}
\left\vert \frac{f\left( a\right) +f\left( b\right) }{2}-\frac{1}{b-a}%
\int_{a}^{b}f\left( x\right) dx\right\vert =\frac{1}{s}\left\vert A\left(
a^{s},b^{s}\right) -\left[ L_{s}\left( a,b\right) \right] ^{s}\right\vert ,
\label{aa}
\end{equation}%
\begin{eqnarray}
&&\left\vert f^{\prime }\left( b\right) \right\vert ^{s}\left[ g_{1}\left(
\alpha \left( s,s\right) \right) +g_{2}\left( \alpha \left( s,s\right)
\right) \right]  \label{bb} \\
&=&b^{^{s\left( s-1\right) }}\frac{4\sqrt{\left( \frac{a}{b}\right)
^{s\left( s-1\right) }}-\ln \left( \frac{a}{b}\right) ^{s\left( s-1\right)
}-2\left( \frac{a}{b}\right) ^{s\left( s-1\right) }+\left( \frac{a}{b}%
\right) ^{s\left( s-1\right) }\ln \left( \frac{a}{b}\right) ^{s\left(
s-1\right) }-2}{\left[ \ln \left( \frac{a}{b}\right) ^{s\left( s-1\right) }%
\right] ^{2}}  \notag \\
&=&\frac{b^{^{s\left( s-1\right) }}}{\ln a^{^{\frac{s-1}{s}}}-\ln b^{^{\frac{%
s-1}{s}}}}\left( \frac{a^{^{s\left( s-1\right) }}-b^{^{s\left( s-1\right) }}%
}{b^{^{s\left( s-1\right) }}}\right) \left[ \frac{a^{^{s\left( s-1\right)
}}+b^{^{s\left( s-1\right) }}}{2b^{^{s\left( s-1\right) }}}-\frac{1}{%
2b^{^{s\left( s-1\right) }}}\frac{a^{^{s\left( s-1\right) }}-b^{^{s\left(
s-1\right) }}}{\ln a^{^{s\left( s-1\right) }}-\ln b^{^{s\left( s-1\right) }}}%
\right]  \notag \\
&=&b^{^{s\left( s-1\right) }}L\left( a^{^{s\left( s-1\right) }},b^{^{s\left(
s-1\right) }}\right) \left[ A\left( a^{^{s\left( s-1\right) }},b^{^{s\left(
s-1\right) }}\right) -\left( 1/2\right) L\left( a^{^{s\left( s-1\right)
}},b^{^{s\left( s-1\right) }}\right) \right] .  \notag
\end{eqnarray}%
From (\ref{aa}) and (\ref{bb}), we have the desired inequality.
\end{proof}

\begin{proposition}
\bigskip Let $0<a<b\leq 1,$ $0<s<1.$ Then%
\begin{equation}
\left\vert A\left( a^{s},b^{s}\right) -\left[ L_{s}\left( a,b\right) \right]
^{s}\right\vert \leq \frac{\left( b-a\right) sb^{^{sq\left( 1-s\right) }}}{%
2\left( p+1\right) ^{1/p}}\left[ L\left( a^{^{sq\left( s-1\right)
}},b^{^{sq\left( s-1\right) }}\right) \right] ^{1/q}  \label{32}
\end{equation}
\end{proposition}

\begin{proof}
\bigskip The proof is obvious from Theorem \ref{t2} applied $f\left(
x\right) =x^{s}/s,$ $x\in \left( 0,1\right] ,$ $0<s<1$ and $q>1.$ Then $%
\left\vert f^{\prime }\left( a\right) \right\vert
=a^{s-1}>b^{s-1}=\left\vert f^{\prime }\left( b\right) \right\vert \geq 1$
and%
\begin{equation}
g_{3}\left( \alpha \left( sq,sq\right) \right) =\frac{a^{^{sq\left(
s-1\right) }}-b^{^{sq\left( s-1\right) }}}{b^{^{sq\left( s-1\right) }}\left(
\ln a^{^{sq\left( s-1\right) }}-\ln b^{^{sq\left( s-1\right) }}\right) }=%
\frac{1}{b^{^{sq\left( s-1\right) }}}L\left( a^{^{sq\left( s-1\right)
}},b^{^{sq\left( s-1\right) }}\right)  \label{cc}
\end{equation}%
From (\ref{cc}), we have the desired inequality.
\end{proof}

\begin{proposition}
\bigskip Let $0<a<b\leq 1,$ $0<s<1$ and $q\geq 1.$ Then%
\begin{equation}
\left\vert A\left( a^{s},b^{s}\right) -\left[ L_{s}\left( a,b\right) \right]
^{s}\right\vert \leq \frac{s\left( b-a\right) }{2}\left( \frac{1}{4}\right)
^{1-\frac{1}{q}}b^{^{s\left( s-1\right) }}\left[ U^{\frac{1}{q}}+V^{\frac{1}{%
q}}\right]  \label{33}
\end{equation}
\end{proposition}

\begin{proof}
\bigskip The proof is obvious from Theorem \ref{t3} applied $f\left(
x\right) =x^{s}/s,$ $x\in \left( 0,1\right] ,$ $0<s<1$ and $q>1.$ Then $%
\left\vert f^{\prime }\left( a\right) \right\vert
=a^{s-1}>b^{s-1}=\left\vert f^{\prime }\left( b\right) \right\vert \geq 1$
and%
\begin{equation}
g_{1}\left( \alpha \left( sq,sq\right) \right) =U=\frac{1}{\ln a^{^{sq\left(
s-1\right) }}-\ln b^{^{sq\left( s-1\right) }}}\left( \frac{1}{b^{^{\frac{%
sq\left( s-1\right) }{2}}}}L\left( a^{^{\frac{sq\left( s-1\right) }{2}%
}},b^{^{\frac{sq\left( s-1\right) }{2}}}\right) -1\right) ,  \label{dd}
\end{equation}%
\begin{eqnarray}
g_{2}\left( \alpha \left( sq,sq\right) \right)  &=&V=\frac{\left( \frac{a}{b}%
\right) ^{2qs\left( s-1\right) }}{\left( \ln a^{^{sq\left( s-1\right) }}-\ln
b^{^{sq\left( s-1\right) }}\right) }\times   \label{ee} \\
&&\left[ 1-\frac{\left( \frac{a}{b}\right) ^{sq\left( s-1\right) }+1}{\left( 
\frac{a}{b}\right) ^{sq\left( s-1\right) }\left( \ln a^{^{\frac{sq\left(
s-1\right) }{2}}}-\ln b^{^{\frac{sq\left( s-1\right) }{2}}}\right) }\right] 
\notag
\end{eqnarray}%
From (\ref{dd}) and (\ref{ee}), we have the desired inequality.
\end{proof}

\end{document}